\documentclass[12pt, reqno]{amsart}
\usepackage{latexsym}
\usepackage{mathtools}
\usepackage{amsmath}
\usepackage{amsthm}
\usepackage{mathrsfs}
\usepackage{lmodern}
\usepackage[T1]{fontenc}
\usepackage{textcomp}
\usepackage[utf8]{inputenc}
\usepackage{color}
\usepackage{amssymb}
\usepackage{enumerate}
\usepackage[abbrev]{amsrefs}
\usepackage{indentfirst}
\usepackage{graphicx}
\usepackage{bmpsize}
\usepackage{hyperref}
\usepackage{enumitem}
\usepackage{enumerate}
\usepackage{url}
\usepackage{autobreak}
\usepackage[justification=justified]{caption}
\usepackage{bm}
\usepackage{here}
\usepackage[labelformat=empty,subrefformat=parens]{subcaption}
\renewcommand{\thefootnote}{} 

\theoremstyle{plain} 
\newtheorem{theorem}{\indent\sc Theorem}[section]
\newtheorem{lemma}[theorem]{\indent\sc Lemma}

\newtheorem{proposition}[theorem]{\indent\sc Proposition}

\theoremstyle{definition} 
\newtheorem{definition}[theorem]{\indent\sc Definition}
\newtheorem{remark}[theorem]{\indent\sc Remark}
\newtheorem{example}[theorem]{\indent\sc Example}

\newtheorem{question}[theorem]{\indent\sc Question}
\newtheorem{problem}[theorem]{\indent\sc Problem}
\newenvironment{sproof}{%
  \proof}{\endproof}
\makeatletter
  
  \@addtoreset{equation}{section}
\makeatother

\newcommand{\cF}{\mathcal{F}}
\newcommand{\cL}{\mathcal{L}}
\newcommand{\cT}{\mathcal{T}}
\newcommand{\cB}{\mathcal{B}}

\newcommand{\stab}[1]{\operatorname{Stab}(#1)}
\newcommand{\oF}{\overrightarrow{F}}
\allowdisplaybreaks[3] 

\newcommand{\Ima}{\operatorname{Im}}

\newcommand\cinput[2]{\lower#1pt\hbox{\input{#2}}}


\begin{document}

\title{Alexander's theorem for stabilizer subgroups of Thompson's group}
\author{Yuya Kodama and Akihiro Takano}
\date{}
\renewcommand{\thefootnote}{\arabic{footnote}}  
\setcounter{footnote}{0} 
\thanks{The first author was supported by JST, the establishment of university
fellowships towards the creation of science technology innovation, Grant Number JPMJFS2139.}

\begin{abstract}
In 2017, Jones studied the unitary representations of Thompson's group $F$ and defined a method to construct knots and links from $F$.
One of his results is that any knot or link can be obtained from an element of this group, which is called Alexander's theorem.
On the other hand, Thompson's group $F$ has many subgroups and it is known that there exist various subgroups which satisfy or do not satisfy Alexander's theorem.
In this paper, we prove that almost all stabilizer subgroups under the natural action on the unit interval satisfy Alexander's theorem.
\end{abstract}

\maketitle

\section{Introduction}
Thompson's groups $F$, $T$, and $V$ were defined by Richard Thompson in 1965. 
These groups were first studied from the motivation of logic. 
Later they have been found to play a central role in various areas, including geometric group theory, and are being actively studied. 

Recently, Jones \cite{jones2017thompson} studied the representations of Thompson's group $F$ and $T$ motivated by the algebraic quantum field theory on the circle.
In this research, he introduced a method to construct knots and links from elements of $F$, that is, he defined the map
$$
\cL \colon F \to \{ \textrm{all knots and links}\}.
$$
Jones showed that this map is surjective.
This is called Alexander's theorem, which originally states a relation with knot theory and the braid group.
Moreover, Jones defined two subgroups of $F$ as the stabilizer subgroups of its unitary representations.
One is called the oriented subgroup $\oF$, and Aiello \cite{aiello2020alexander} proved that the map $\cL$ restricted to $\oF$ is also surjective.
In contrast, the other is called the 3-colorable subgroup $\cF$ \cite{jones2018nogo}, and the authors \cite{kodama20223} proved that all knots and links obtained from non-trivial elements of $\cF$ are $3$-colorable.
Namely, the map $\cL$ restricted to $\cF$ is not surjective.
To the best of the authors' knowledge, this subgroup is the first example of the one which does not satisfy Alexander's theorem but produces non-trivial knots or links.
Therefore we can consider the following problem:

\begin{problem}
Let $G$ be a subgroup of $F$.
Determine whether the restricted map
$$
\cL |_{G} \colon G \to \{ \textrm{all knots and links}\}
$$
is surjective or not.
More specifically, determine the image of this map.
\end{problem}

In this paper, we focus on the stabilizer subgroups under the natural action on $[0, 1]$ and show the following result:
\begin{theorem}\label{Alexander_stabilizer}
Let $r \in (0, 1)$ with $r \neq 1/2$. 
Then for any knot or link $L$, there exists $g \in \stab{r}$ such that $\cL (g)=L$. 
\end{theorem}
Note that we have $\stab{0}=\stab{1}=F$. 
If $r=1/2$, then the map $\cL_{\stab{1/2}}$ is not surjective. 
See Remark \ref{1/2_case} for details. 

This paper is organized as follows:
in Section \ref{section_preliminary}, we first summarize the definition of Thompson's group $F$.
We then explain the method of constructing knots and links from elements of $F$, and introduce a binary operation of $F$ which is called the attaching operation.
This operation corresponds to the connected sum of links, and is therefore useful for studying Jones's map $\cL$. 
In Section \ref{proof_Alexander}, we first prepare some elements of $F$ and operations, and then we show the main theorem.
In Section \ref{other_subgroup}, we give other examples of subgroups which satisfy or do not satisfy Alexander's theorem.
\section{Preliminaries} \label{section_preliminary}
\subsection{Thompson's group $F$}
In this section, we summarize the definition of Thompson's group $F$ and its properties. 
See \cite{cannon1996intro} for details. 
\begin{definition}
\textbf{Thompson's group $F$} is the group consisting of homeomorphisms on the closed unit interval $[0, 1]$ which satisfy the following: 
\begin{enumerate}
\item they are piecewise linear maps and preserve the orientation, 
\item in each linear part, its slope is a power of $2$, and
\item each breakpoint is in $\mathbb{Z}[1/2] \times \mathbb{Z}[1/2]$. 
\end{enumerate}
\end{definition}
Hence $F$ has a natural right action on $[0, 1]$, and we will consider the stabilizer subgroups for this action. 

It is known that there exist several (equivalent) definitions for $F$. 
We also use binary trees to represent elements of $F$. 
Let $\cT$ be the set of all pairs of (rooted) binary trees whose numbers of leaves are the same. 
We define an equivalence relation on $\cT$ as follows: 
let $(T_+, T_-)$ be in $\cT$ with $n$ leaves. 
We label the leaves of $T_+$ and $T_-$ with $0, \dots, n-1$ from left to right. 
We assume that there exists $i \in \{0, \dots n-2\}$ such that two leaves labeled $i$ and $i+1$ have a common parent in both $T_+$ and $T_-$. 
Then we can obtain a new element $(T_+^\prime, T_-^\prime) \in \cT$ by removing the minimal binary tree containing leaves labeled $i$, $i+1$ and their parent from both $T_+$ and $T_-$. 
A binary tree such as this minimal tree, consisting of three vertices and two edges, is termed a \textbf{caret}.
We call this operation and its inverse one \textbf{reduction} and \textbf{insertion}, respectively. 
We define an equivalence relation $\sim$ on $\cT$ as the one generated by these operations. 
We call each element of $\cT$ a \textbf{tree pair}, and define a tree pair $(T_+, T_-)$ is \textbf{reduced} if there exists no leave in $T_+$ and $T_-$ such that we can not reduce any caret. 
It is known that for any $g \in F$, there exists a unique tree pair $(T_+, T_-)$ such that $(T_+, T_-)$ is reduced and is a representative of $g$ \cite[\S 2]{cannon1996intro}. 

We define an operation $\times$ on $\cT/{\sim}$ as follows: 
Let $A, B$ in $\cT/{\sim}$, and $(A_+, A_-)$ and $(B_+, B_-)$ be its representative, respectively. 
Then there exists a binary tree $C$ such that $(A_+, A_-) \sim (A_+^\prime, C)$ and $(B_+, B_-) \sim (C, B_-^\prime)$ hold. 
We define $A \times B$ to be the equivalence class of $(A_+^\prime, B_-^\prime)$. 
This operation does not depend on the choice of $C$. 
The identity element is the equivalence class of $(T, T)$ where $T$ is a binary tree, and for the equivalence class of $(T_+, T_-)$, its inverse element is the equivalence class of $(T_-, T_+)$. 
Hence $(\cT/{\sim}, \times)$ is a group. 

As we describe below, the group $(\cT/{\sim}, \times)$ and $F$ are isomorphic. 
This is a well-known fact, but since we use this identification repeatedly, we write a sketch of proof for the reader's convenience. 
\begin{proposition}[{\cite[\S 2]{cannon1996intro}}] \label{proposition_isomorphism_F}
The group $(\cT/{\sim}, \times)$ is isomorphic to Thompson's group $F$. 
\begin{sproof}
Let $A$ be a binary tree with $n$ leaves. 
We decompose the closed interval $[0, 1]$ by $A$ as follows: 
first, we assign $[0, 1]$ to the root of $A$. 
Next, if a parent has the interval $[a, b]$, then we assign $[a, (a+b)/2]$ and $[(a+b)/2, b]$ to the left and right children, respectively. 
By iterating this process, we get subintervals $[a_0, a_1], [a_1, a_2], \dots, [a_{n-1}, a_n]$ where $a_0=0$ and $a_n=1$ from leaves of $A$.

Let $X$ be in $\cT/{\sim}$, and $(A, B)$ be its representative. 
Let $[a_0, a_1], [a_1, a_2], \dots, [a_{n-1}, a_n]$ and $[b_0, b_1], [b_1, b_2], \dots, [b_{n-1}, b_n]$ be the subintervals obtained from $A$ and $B$, respectively. 
By mapping each subinterval $[a_i, a_{i+1}]$ to $[b_i, b_{i+1}]$ linearly, we get a map $f$ in $F$ from $\cT/{\sim}$. 
This correspondence implies that $(\cT/{\sim}, \times)$ is isomorphic to $F$. 
\end{sproof}
\end{proposition}
Let $T$ be a binary tree and $i$ be its leaf. 
Since $T$ is a tree, there exists a unique path from the root to $i$. 
For each caret, by setting $0$ when it bifurcates to the left and $1$ when it bifurcates to the right, we identify the binary word with the path (hence with the leaf $i$). 
For a leaf $i$, let $[a_i, a_{i+1}]$ be the corresponding subinterval defined in Proposition \ref{proposition_isomorphism_F} and $w_1 \cdots w_n$ be the binary word  corresponding to $i$. 
Note that we have $a_i=\Sigma_{i=1}^{n} w_i/2^i$. 

There exist two well-known presentations for $F$: 
\begin{align*}
F &\cong \langle x_0, x_1, \dots \mid x_i^{-1}x_jx_i=x_{j+1} (i<j) \rangle \\
&\cong \langle x_0, x_1 \mid [x_0x_1^{-1}, x_0^{-1}x_1x_0], [x_0x_1^{-1}, x_0^{-2}x_1x_0^2] \rangle, 
\end{align*}
where the generating set $\{x_0, x_1, \dots\}$ are illustrated in Figure \ref{generators_x0x1x2}, and $[a, b]$ denotes the commutator $aba^{-1}b^{-1}$. 
We use these generators in Section \ref{other_subgroup}. 
\begin{figure}[tbp]
\begin{center}
\includegraphics[width=\linewidth]{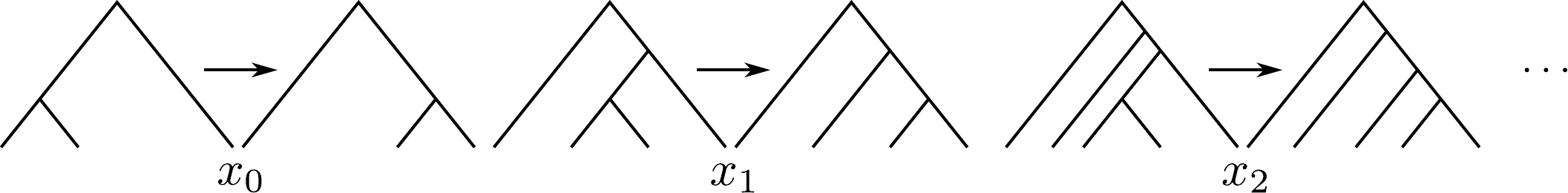}
\end{center}
\caption{The generators of $F$. }
\label{generators_x0x1x2}
\end{figure}
\subsection{Jones' construction} \label{jones}
In this section, we explain the method to construct knots and links from elements of $F$ based on \cite{jones2019thompson}.
Let $(T_+, T_-)$ be a tree pair.


For any region, including the unbounded one, of $(T_+, T_-)$, there exist exactly two carets, one in $T_+$ and the other in $T_-$.
Then we connect roots of such carets by an edge in the region, and obtain the plane graph $\cB(T_+, T_-)$ associated with $(T_+, T_-)$.
Figure~\ref{step1} is an example of the graph $\cB(T_+, T_-)$.

\begin{figure}[tbp]
\begin{center}
\includegraphics[height=120pt]{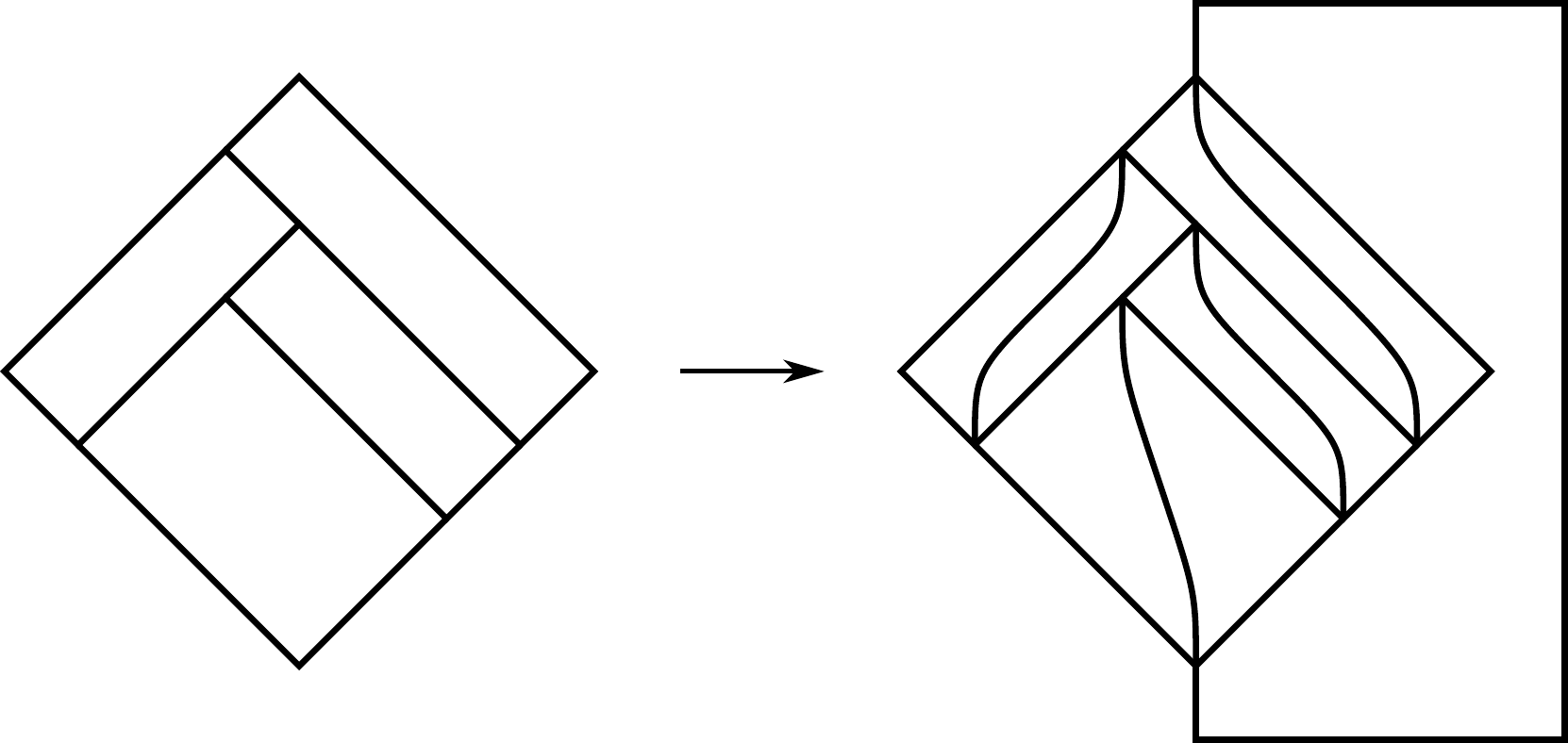}
\end{center}
\caption{The plane graph $\cB(T_+, T_-)$ associated with $(T_+, T_-)$.}
\label{step1}
\end{figure}


Then since each bifurcation has exactly four edges in $\cB(T_+, T_-)$, this graph can be regarded as a link projection.
Hence we obtain the link diagram $\cL(T_+, T_-)$ by turning each vertex into a crossing with the rules \cinput{5}{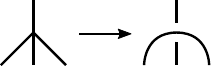} or \cinput{5}{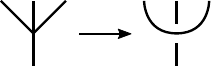}.
Figure~\ref{step2} is an example of the link diagram $\cL(T_+, T_-)$.

\begin{figure}[tbp]
\begin{center}
\includegraphics[height=120pt]{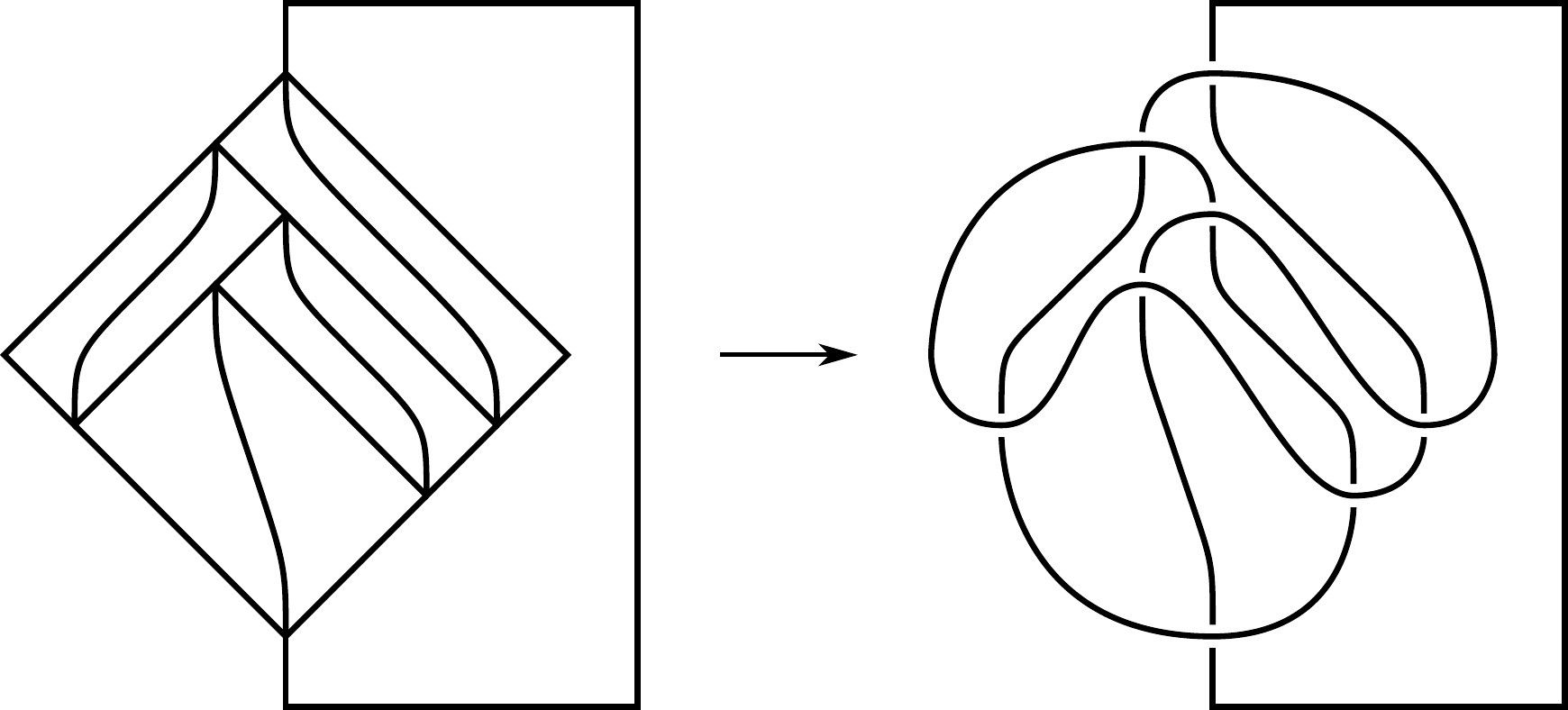}
\end{center}
\caption{The link diagram $\cL(T_+, T_-)$.
This gives the trefoil knot $3_1$.}
\label{step2}
\end{figure}

This construction is defined for any tree pair.
However, the link diagram of a non-reduced tree pair is different from that of the equivalent reduced tree pair:
let $(T'_+, T'_-)$ be a non-reduced tree pair obtained by inserting carets into a leaf of the reduced tree pair $(T_+, T_-)$.
Then applying Jones' construction to $(T'_+, T'_-)$, we have $\cL(T'_+, T'_-) = \cL(T_+, T_-) \sqcup \lower-1pt\hbox{$\bigcirc$}$; see Figure~\ref{caret_triv}.
On the other hand, each element $g \in F$ has a unique reduced tree pair.
Therefore we can consider a map
$$
\cL \colon F \to \{ \textrm{all knots and links}\}
$$
defined by $\cL(g)$ as the link diagram obtained from its reduced tree pair.
Jones \cite[Theorem 5.3.1]{jones2017thompson} proved that this map is surjective, which is called \textbf{Alexander's theorem}.
However, it is easy to see that this is not injective.

\begin{figure}[tbp]
\begin{center}
\includegraphics[height=50pt]{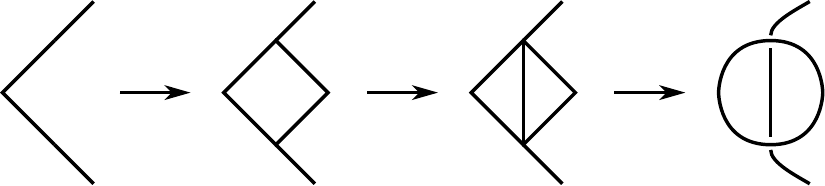}
\end{center}
\caption{A trivial link component by inserting carets.}
\label{caret_triv}
\end{figure}

At the end of this section, we describe a relation between Jones' construction and a geometric modification of links, called the connected sum.

\begin{definition} \label{attaching}
The \textbf{$i$-th attaching operation} $\circ_i \colon \cT \times \cT \to \cT$ is defined as follows:
let $(A_+, A_-), (B_+, B_-)$ be two tree pairs (not necessarily reduced).
Then the new tree pair $(A_+, A_-) \circ_i (B_+, B_-)$ is defined as the one obtained by attaching the root of $B_+$ (resp.~$B_-$) to the $i$-th leaf of $A_+$ (resp.~$A_-$); see below.
\begin{align*}
(A_+, A_-) \circ_i (B_+, B_-) \coloneqq \cinput{30}{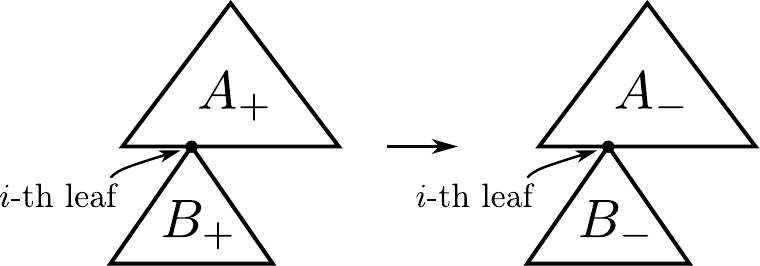}
\end{align*}
\end{definition}
\begin{remark}
In \cite[Section 2]{golan2019divergence}, they defined a similar notion by using branches of tree pairs. 
We briefly review only the case of $F$. 
Let $(T_+, T_-)$ be a reduced tree pair of $g \in F$. 
Let $u$ be the path from the root to the $i$-th leaf of $T_+$. 
Similarly, let $v$ be the $i$-th path of $T_-$. 
Then the pair of paths $u\to v$ is called a \textbf{branch} of $(T_+, T_-)$. 
Let $T_v$ be a minimal binary tree such that $T_v$ contains $v$ as a (binary) path from the root to a leaf and $j$ be the number corresponding to this leaf. 
For reduced tree pair $(H_+, H_-)$ of $h \in F$, we define $h_{[v]} \coloneqq (T_v, T_v) \circ_j (H_+, H_-)$. 
Then the reduced tree diagram of $g h_{[v]}$ (this appears in \cite[Lemma 2.6]{golan2019divergence}) and $(T_+, T_-) \circ_i (H_+, H_-)$ are the same. 
\end{remark}

Similarly to Jones' map $\cL$, by choosing the reduced tree pair, the $i$-th attaching operation $\circ_i \colon F \times F \to F$ is well-defined.

By direct computations, we easily see the following lemma.

\begin{lemma}\label{lemma_connected_sum}
Let $g, h$ be two elements of $F$.
Then the link $\cL(g \circ_i h)$ is the connected sum of $\cL(g)$ and $\cL(h)$. 
In particular, if $\cL(g)$ $($resp.~$\cL(h)$$)$ is the unknot, then $\cL(g \circ_i h)$ coincides with $\cL(h)$ $($resp.~$\cL(g)$$)$.
\end{lemma}
\section{Proof of Theorem \ref{Alexander_stabilizer}} \label{proof_Alexander}
At first, in order to prove Theorem \ref{Alexander_stabilizer}, we introduce \textbf{basic elements} in $F$ which stabilize some intervals.

\begin{definition} \label{basic_element}
Let $u \in \{ 000, 00100, 00101, 0011, 0100, 0101, 0110 \}$, then the basic element $g_u = (T^u_+, T^u_-) \in F$ is defined as follows:
\begin{align*}
g_{000} &\coloneqq \cinput{13}{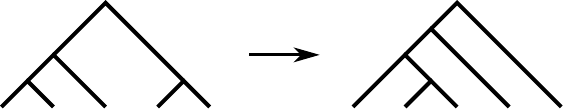},\\
g_{00100} &\coloneqq \cinput{18}{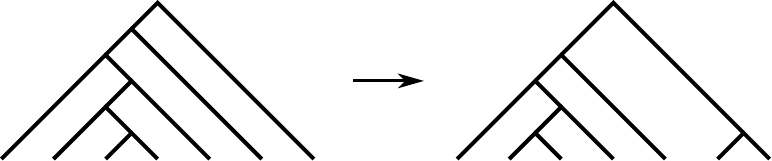},\\
g_{00101} &\coloneqq \cinput{23}{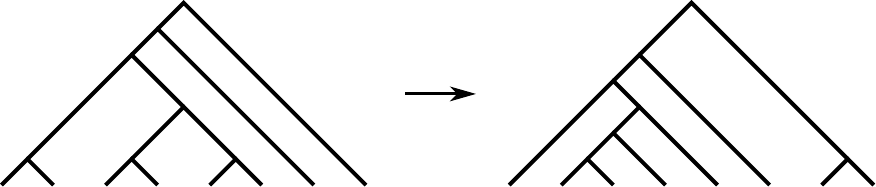},\\
g_{0011} &\coloneqq \cinput{18}{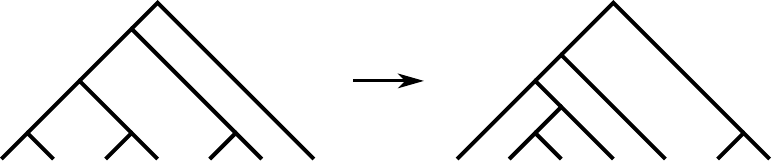},\\
g_{0100} &\coloneqq \cinput{15}{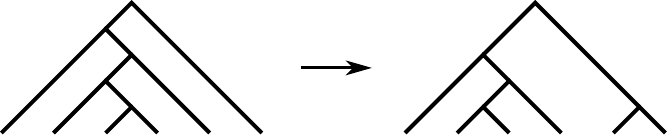},\\
g_{0101} &\coloneqq \cinput{18}{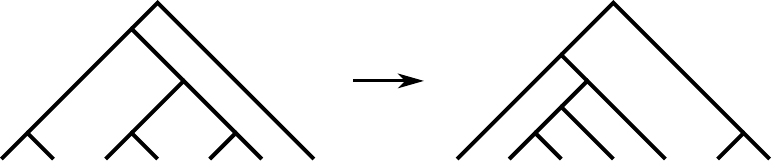},\ \ \textrm{and}\\
g_{0110} &\coloneqq \cinput{18}{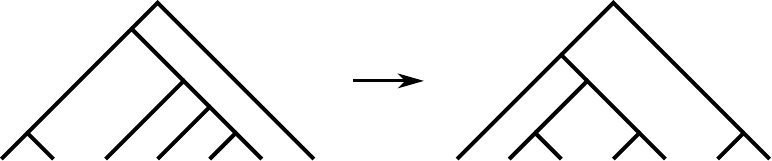}.
\end{align*}
\end{definition}
For example, since the $0$-th paths of $T^{000}_+$ and $T^{000}_-$ are both $000$, the element $g_{000}$ stabilizes the interval $[0, 1/8]$.
Similarly, each element $g_u$ stabilizes the interval corresponding to the binary word $u$.
Moreover, all of the links $\cL (g_u)$ are the unknots; see Figure \ref{basic_link} for instance.

\begin{figure}[!tbp]
\centering
\subfloat[]{\includegraphics{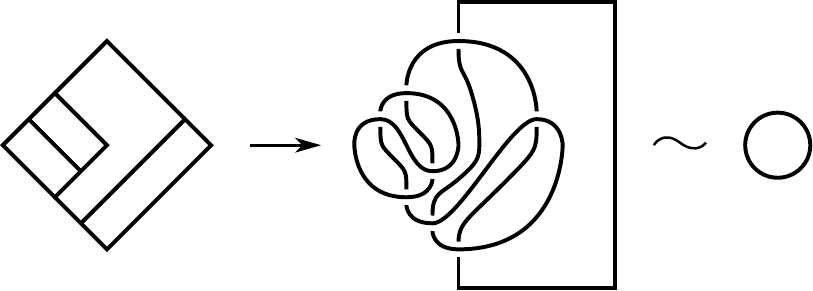}
\label{}}
\\
\subfloat[]{\includegraphics{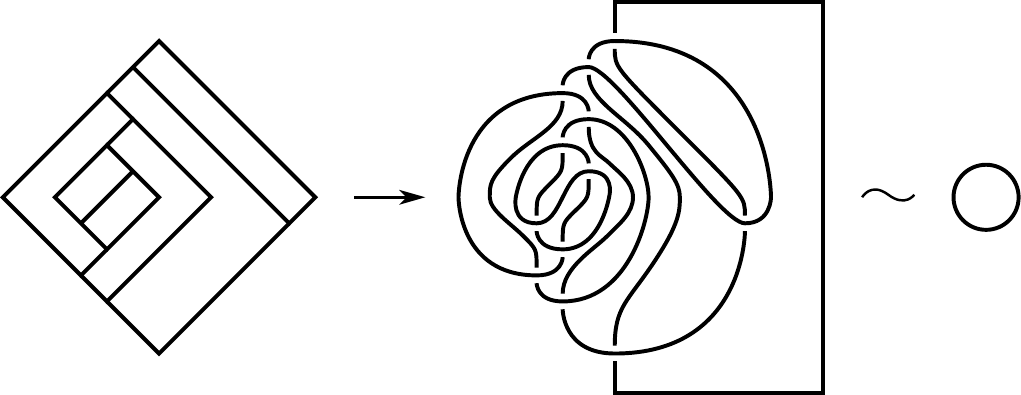}
\label{}}
\vspace{-20pt}
\caption{The links $\cL(g_{000})$ and $\cL(g_{00100})$ are the unknots.}
\label{basic_link}
\end{figure}

Next, we define the following operations of tree pairs.

\begin{definition}
Let $(T_+, T_-)$ be a tree pair (not necessarily reduced), then two operations $A_{00}, A_{01} \colon \cT \to \cT$ are defined as described below:
\begin{align*}
A_{00} (T_+, T_-) &\coloneqq \cinput{28}{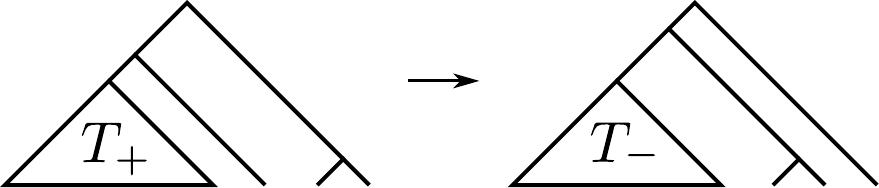},\\
A_{01} (T_+, T_-) &\coloneqq \cinput{33}{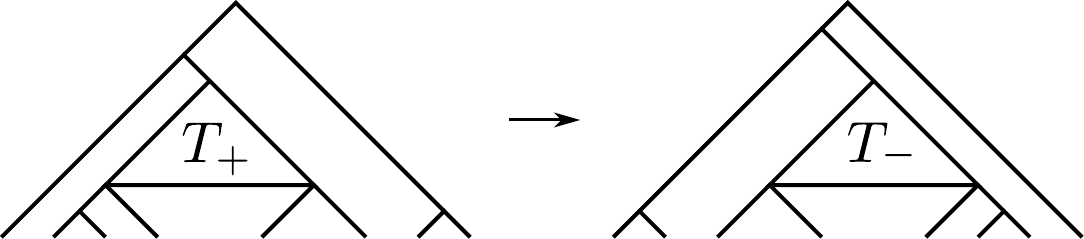}.
\end{align*}
Moreover, by flipping left and right, the operations $A_{11}, A_{10} \colon \cT \to \cT$ are defined in the same way.
By choosing the reduced tree pair, these operations are defined over $F$.
\end{definition}

\begin{remark} \label{A00_attaching}
The operation $A_{00}$ can be considered as the $0$-th attaching operation with the given $g$, that is, we have
\begin{align*}
A_{00} (g) = \left(\ \cinput{15}{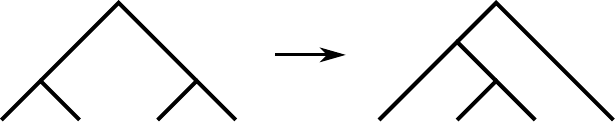}\ \right) \circ_0 g.
\end{align*}
\end{remark}

\begin{lemma} \label{lemma_00operation}
For any element $g = (T_+, T_-) \in F$, all of the links $\cL(A_{00}(g))$, $\cL(A_{01}(g))$, $\cL(A_{10}(g))$ and $\cL(A_{11}(g))$ coincide with $\cL(g)$.
\begin{proof}
It is sufficient to show for the cases $A_{00}$ and $A_{01}$.
By remark \ref{A00_attaching}, $A_{00}$ can be considered as the $0$-th attaching operation, and the first element above gives the unknot.
Therefore, Lemma \ref{lemma_connected_sum} implies that the link $\cL(A_{00}(g))$ is exactly $\cL(g)$.
Also the case of $A_{01}$ is proved as in Figure \ref{A01link}.
\begin{figure}[tbp]
\begin{center}
\includegraphics{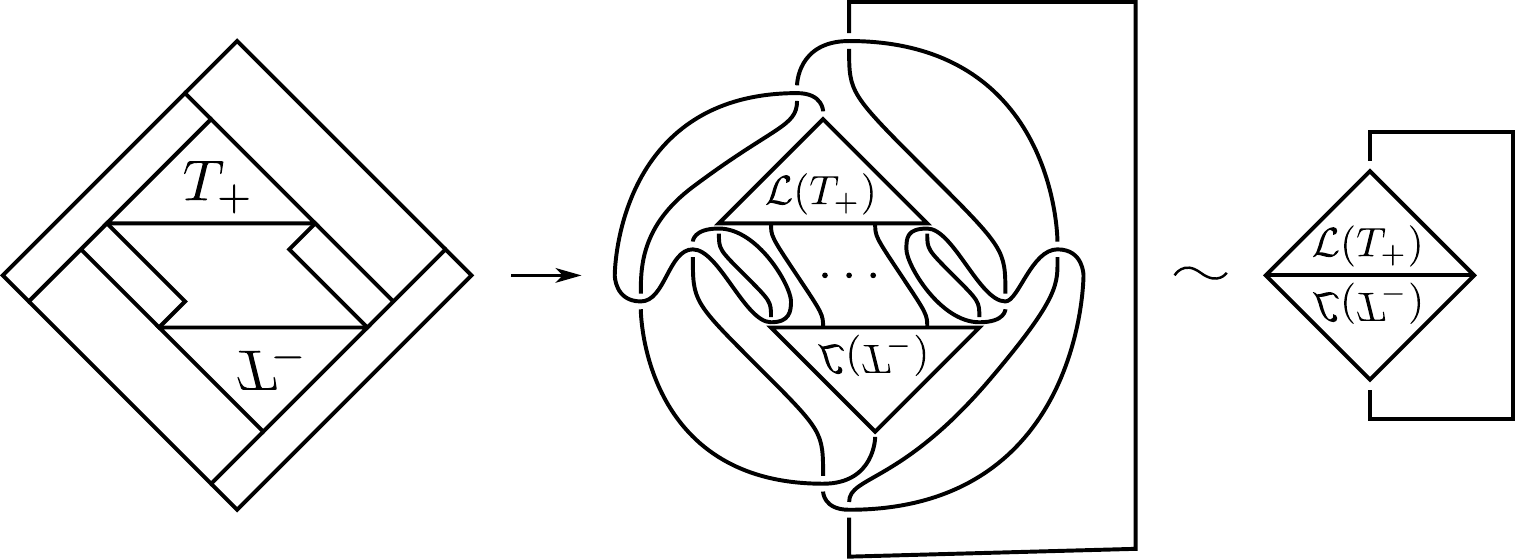}
\end{center}
\caption{The link $\cL (A_{01}(T_+, T_-))$, where $\cL(T_+)$ (resp.~$\cL(T_-)$) is the tangle obtained naturally from the binary tree $T_+$ (resp.~$T_-$) by Jones's construction.}
\label{A01link}
\end{figure}
\end{proof}
\end{lemma}

\begin{proof}[Proof of Theorem $\ref{Alexander_stabilizer}$]
Let $r$ in $(0, 1)$ with $r \neq 1/2$. 
We will construct $f \in \stab{r}$ such that $f$ is not the identity map and $\cL(f)$ is the unknot. 
By Lemma \ref{lemma_connected_sum}, the existence of the element $f$ implies that Alexander's theorem holds for $\stab{r}$. 
\begin{lemma} \label{lemma_stabilizer1}
If $r \in (0, 7/16] \cup [9/16, 1)$, then there exists an element $f \in \stab{r}$ such that $f$ is not the identity map and $\cL(f)$ is the unknot. 
\begin{proof}
If $r \in (0, 7/16]$, then it is clear from the definition of the basic elements defined in the previous section. 
If $r \in [9/16, 1)$, then by flipping left and right of the basic elements, we get the elements $g_{1001}, g_{1010},  g_{1011},  g_{1100},  g_{11010},  g_{11011}$ and $ g_{111}$. 
One of these elements is desired one. 
\end{proof}
\end{lemma}
We call $g_{1001}, g_{1010},  g_{1011},  g_{1100},  g_{11010},  g_{11011}$ and $ g_{111}$ also basic elements. 
\begin{lemma} \label{lemma_stabilizer2}
If $r \in [7/16, 9/16]$ with $r \neq 1/2$, then there exists an element $f \in \stab{r}$ such that $f$ is not the identity map and $\cL(f)$ is the unknot. 
\begin{proof}
Let $r \in [7/16, 1/2)$ and $0.0111\alpha$ be a binary representation of $r$. 
Since $r \neq 1/2$, $\alpha$ is not $111 \cdots$. 
Let $\alpha=\underbrace{11 \cdots 1}_{i}0\alpha ^\prime$. 
If $i=2k$ ($k \geq 0$), then we have
\begin{align*}
0.0111\alpha=0. 01\underbrace{1\cdots1}_{2k}110 \alpha^\prime. 
\end{align*}
Hence for a basic element $g$ which is either $g_{1100}, g_{11010}$, or $g_{11011}$, $A_{01}(A_{11}^k(g))$ is in $\stab{r}$. 
If $i=2k+1$ ($k \geq 0$), then we have
\begin{align*}
0.0111\alpha=0. 01\underbrace{1\cdots1}_{2k}1110 \alpha^\prime. 
\end{align*}
Therefore, $A_{01}(A_{11}^k(g_{111}))$ is in $\stab{r}$. 
Indeed, it is clear from the observation of the difference in the action of 
\begin{align*}
r=0. 01\underbrace{1\cdots1}_{2k}1110 \alpha^\prime
\end{align*}
and
\begin{align*}
0. 01\underbrace{1\cdots1}_{2k}1111 \alpha^\prime
\end{align*}
by $A_{01}(A_{11}^k(g_{111}))$. 
The proof is similar when $r$ is in $(1/2, 9/16]$. 
\end{proof}
\end{lemma}
By Lemmas \ref{lemma_stabilizer1} and \ref{lemma_stabilizer2}, for every $r$ except $1/2$,  we have that Alexander's theorem holds for $\stab{r}$. 
\end{proof}
\begin{example}
Let $r=1/3=0. 010101\cdots$. 
Then $g_{0101}$ is a nontrivial element in $\stab{1/3}$. 
Observe Figure \ref{0101fix_into_example} that Alexander's theorem for $\stab{1/3}$ holds. 
\end{example}
\begin{figure}[tbp]
\begin{center}
\includegraphics[width=0.9\linewidth]{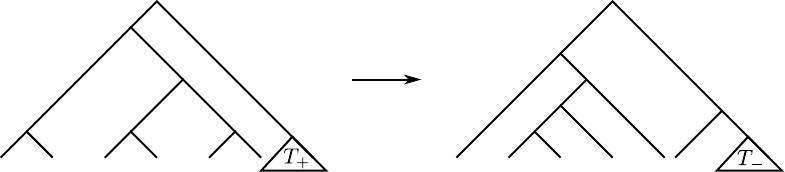}
\end{center}
\caption{An element in $\stab{1/3}$, where $(T_+, T_-)$ is an arbitrary tree pair. }
\label{0101fix_into_example}
\end{figure}
\begin{remark} \label{1/2_case}
If $r = 1/2$, then any tree pair of $\stab{1/2}$ is described as in Figure \ref{tree_pair_half}.
\begin{figure}[tbp]
\begin{center}
\includegraphics{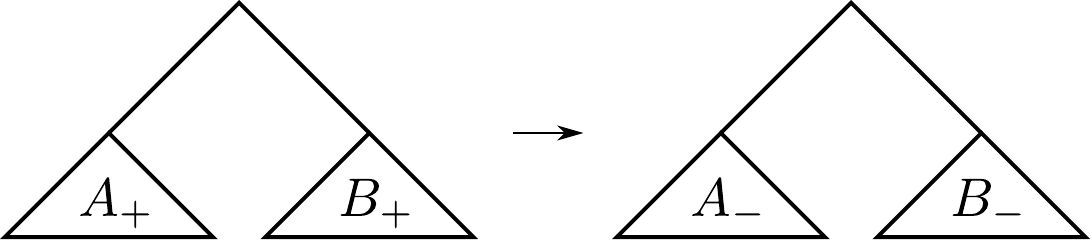}
\end{center}
\caption{An element of $\stab{1/2}$, where $(A_+, A_-)$ and $(B_+, B_-)$ are some tree pairs.}
\label{tree_pair_half}
\end{figure}
Therefore, this element produces a split link, and thus Alexander's theorem does not hold.
More precisely, we have
\begin{align*}
\Ima \cL |_{\stab{1/2}} = \left\{ L \sqcup \lower-1pt\hbox{$\bigcirc$} \mid L\ \textrm{is any knot or link} \right\}.
\end{align*}
In particular, since the action of $F$ on $\mathbb{Z}[1/2]$ is transitive, whether Alexander's theorem holds or not is not preserved under inner automorphisms.  
\end{remark}
In general, let $U$ be a subset of $(0,1)$.
Then we can consider the problem of Alexander's theorem for the stabilizer subgroup $\stab{U}$.
For example, from Remark \ref{1/2_case}, if $U$ contains $1/2$, then $\stab{U}$ does not satisfy Alexander's theorem.
Moreover, we see that $\stab{\{1/4, 1/8\}}$ does not satisfy the theorem either.
On the other hand, since $g_{0101}$ stabilizes $[5/16, 3/8]$, the subgroup $\stab{U}$ satisfies Alexander's theorem, where $U$ is any subset of this subinterval.
\begin{question}
Let $U$ be a subset of $(0,1)$.
What is the necessary and sufficient condition for $\stab{U}$ to satisfy Alexander's theorem?
\end{question}
\section{Other subgroups of Thompson's group $F$} \label{other_subgroup}
In this section, we give other examples which we know whether Alexander's theorem holds or not. 
\begin{example}[the commutator subgroup $F^\prime$] \label{example_commutator}
Note that we have
\begin{align*}
F^\prime = \{f \in F \mid \text{$f$ is the identity in the neighborhoods of $0$ and $1$} \}. 
\end{align*}
Hence the tree pair $X$ in Figure \ref{Alex_commutator} is in $F^\prime$ for any $(T_+, T_-)$. 
\begin{figure}[tbp]
\begin{center}
\includegraphics[width=0.7\linewidth]{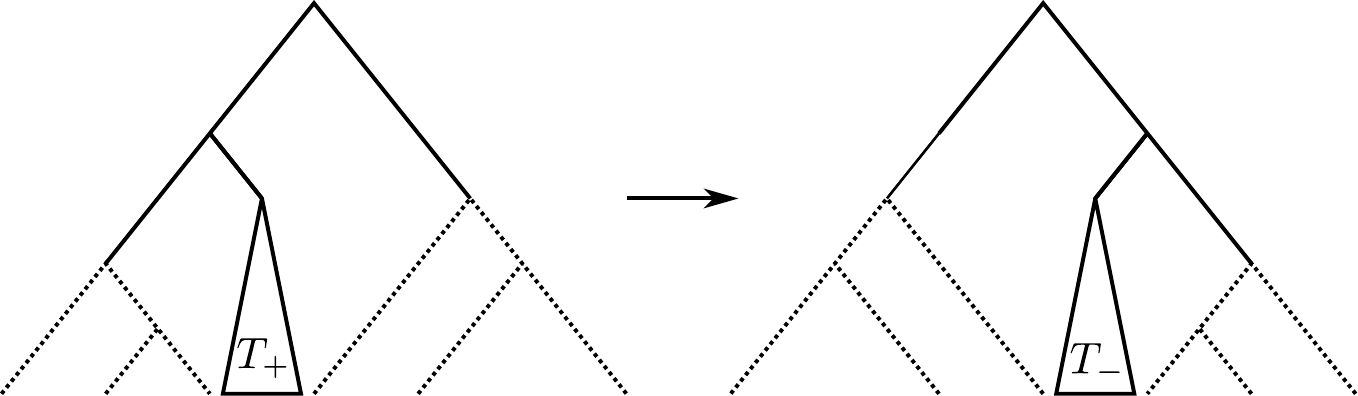}
\end{center}
\caption{An element $X$ in $F^\prime$ where $(T_+, T_-)$ is reduced tree pair in $F$. The dotted binary trees correspond to the binary tree of $x_0$. }
\label{Alex_commutator}
\end{figure}
See \cite[Theorem 4.1]{cannon1996intro} for details. 
Since $\cL(x_0)$ is the unknot, by Lemma \ref{lemma_connected_sum}, the two tree pairs $X$ and $(T_+, T_-)$ give the same link. 
Therefore Alexander's theorem holds for $F^\prime$. 

In particular, if a subgroup $G$ of $F$ contains $F^\prime$ as a subgroup, then Alexander's theorem holds for $G$. 
For instance, Alexander's theorem holds for any rectangular subgroup. 
Let $a, b$ in $\mathbb{N}$. 
We define the rectangular subgroup $K_{(a, b)}$ as
\begin{align*}
K_{(a, b)} \coloneqq \{f \in F \mid \log_2 f^\prime(0) \in a\mathbb{Z}, \log_2 f^\prime(1) \in b\mathbb{Z} \}. 
\end{align*}
This group was introduced in \cite{bleak2007finite} to study the classification of the finite index subgroups of $F$. 
\end{example}
\begin{example}[the subgroup isomorphic to the restricted wreath product $F \wr \mathbb{Z}$] \label{example_FwrZ}
Wu and Chen \cite[Theorem 3.1]{wu2014distortion} showed that the subgroup $H_1$ generated by $x_0, x_1^2x_2^{-1}x_1^{-1}$ and $x_1x_2^2 x_3^{-1}x_2^{-1}x_1^{-1}$ is isomorphic to $F \wr \mathbb{Z}$ and it is undistorted. 

We note that the subgroup (denoted $\Phi_0$ in \cite{wu2014distortion} and introduced in \cite[Lemma 19]{guba1999subgroups}) generated by $x_1^2x_2^{-1}x_1^{-1}$ and $x_1x_2^2 x_3^{-1}x_2^{-1}x_1^{-1}$ is 
\begin{align*}
\{f \in F \mid \text{$f$ is identity on $[0, 1/2]$ and $[3/4, 1]$}\}. 
\end{align*}
Therefore the tree pair $(A_+, A_-)$ illustrated in Figure \ref{Alex_FwrZ} is in this group for any $(T_+, T_-)$, and this group is isomorphic to $F$ (this was stated in \cite[Lemma 19]{guba1999subgroups}). 
\begin{figure}[tbp]
\begin{center}
\includegraphics[width=0.5\linewidth]{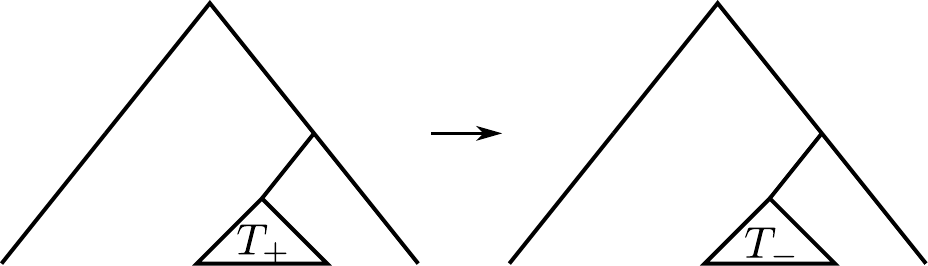}
\end{center}
\caption{An  tree pair $(A_+, A_-)$ where $(T_+, T_-)$ is reduced tree pair in $F$. }
\label{Alex_FwrZ}
\end{figure}
By considering $x_0 \times (A_+, A_-)=x_0 \circ_1 (T_+, T_-)$, it follows by Lemma \ref{lemma_connected_sum} that Alexander's theorem holds for $H_1$. 
\end{example}
\begin{example}[the oriented subgroup $\oF$] \label{example_oriented}
In \cite{aiello2020alexander}, Aiello showed that any (oriented) link can be obtained from $\oF$. 
Golan and Sapir \cite[Theorem 1]{golan2017jones} showed that $\oF$ is generated by $x_0x_1, x_1x_2$, and $x_2x_3$, and is isomorphic to the Brown--Thompson group $F(3)$. 
\end{example}
The following subgroups are examples where Alexander's theorem does not hold. 
\begin{example}[the group generated by $x_0$] \label{ex_x0}
Since $\cL(x_0^k)$ and $\cL(x_0^{-k})$ are the same link, assume that $k \geq 0$. 
Then $\cL(x_0)$ and $\cL(x_0^2)$ are the unknots, and $\cL(x_0^3)$ is the 2-component unlink.
In general, $\cL(x_0^k)$ is the same as $\cL(x_0^{k-3})$ for any $k \geq 4$; see Figure \ref{x0_power}.
\begin{figure}[tbp]
\begin{center}
\includegraphics{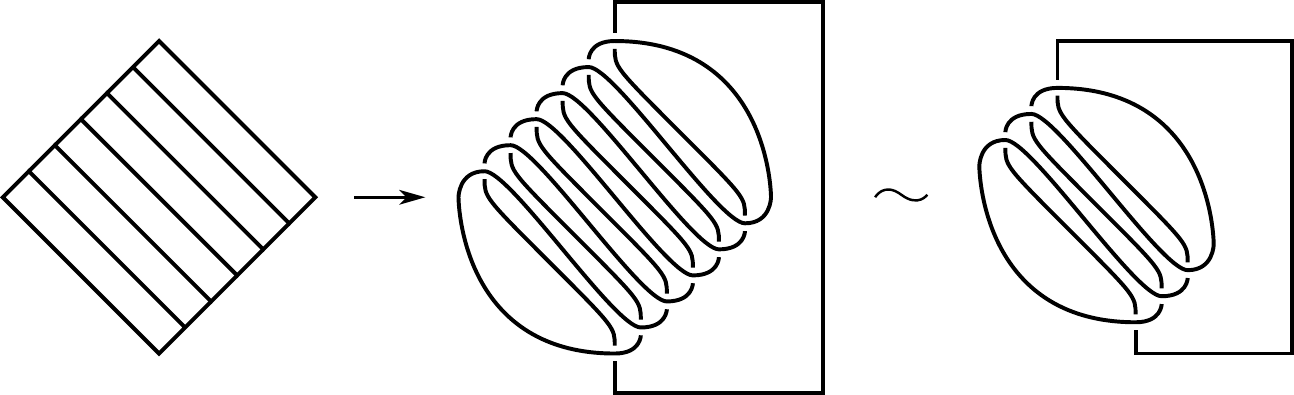}
\end{center}
\caption{$\cL(x_0^k)$ is the same as $\cL(x_0^{k-3})\ (k = 5)$.}
\label{x0_power}
\end{figure}
\end{example}
\begin{example}[the subgroup isomorphic to the restricted wreath product $\mathbb{Z} \wr \mathbb{Z}$] \label{example_ZwrZ}
Let $H_2$ be the group generated by $x_0$ and $x_1^2x_2^{-1}x_1^{-1}$. 
This group is isomorphic to $\mathbb{Z} \wr \mathbb{Z}$ and undistorted \cite[Theorem 3.1]{cleary2006distortion}. 
Let $a \coloneqq x_1^2x_2^{-1}x_1^{-1}$ and $t \coloneqq x_0$. 
We define $a_n$ to be $t^{-n} a t^n$ where $n$ is in $\mathbb{Z}$. 
Similar to \cite{cleary2006distortion}, for a word $w$ in $H_2$, we consider its (right-first) normal form
\begin{align*}
a_{i_1}^{e_1} a_{i_2}^{e_2} \cdots a_{i_k}^{e_k} a_{-j_1}^{f_1} a_{-j_2}^{f_2} \cdots a_{-j_l}^{f_l} t^m, 
\end{align*}
where $i_k > \cdots i_2 > i_1 \geq 0$ and $j_l> \cdots j_2>j_1 >0$ and $e_i, f_j \neq 0$. 
If $m=0$, then the tree pair of $w$ is as in Figure \ref{ZwrZ_1}. 
\begin{figure}[tbp]
\begin{center}
\includegraphics{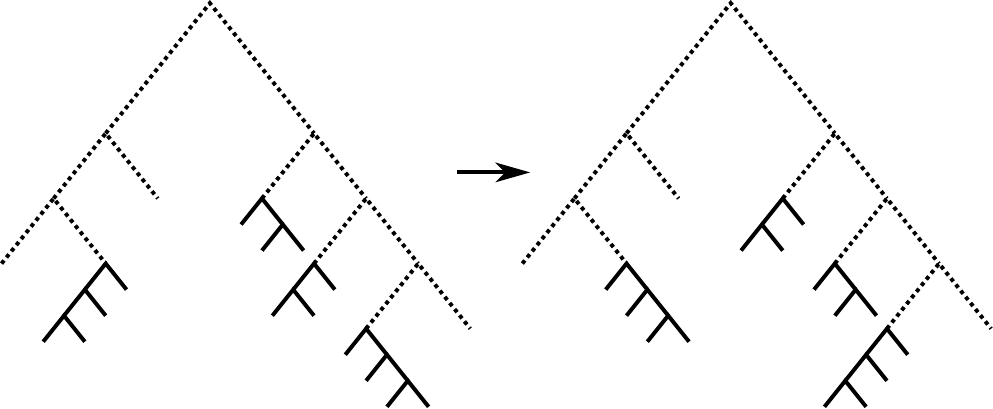}
\end{center}
\caption{A tree pair of $a_0^{-1}a_1^{1}a_2^{-2}a_{-2}^{2}$. The dotted tree pair is a representative of the identity element of $F$. }
\label{ZwrZ_1}
\end{figure}
Since the tree pair of $w$ is given by attaching the tree pairs of $x_0^{e_1}, \dots, x_0^{e_k}, x_0^{f_1}, \dots, x_0^{f_l}$ to each certain leaves of the (trivial) tree pair, by Lemma \ref{lemma_connected_sum} and Example \ref{ex_x0}, $\cL(w)$ is an unlink. 
Observe that the number of components can be increased as desired. 
If $m \neq 0$, then we also obtain an unlink since the tree pair before attaching $x_0^i$s is equivalent to the tree pair of $x_0^m$. 
See Figure \ref{ZwrZ_2} for example. 
\begin{figure}[tbp]
\begin{center}
\includegraphics{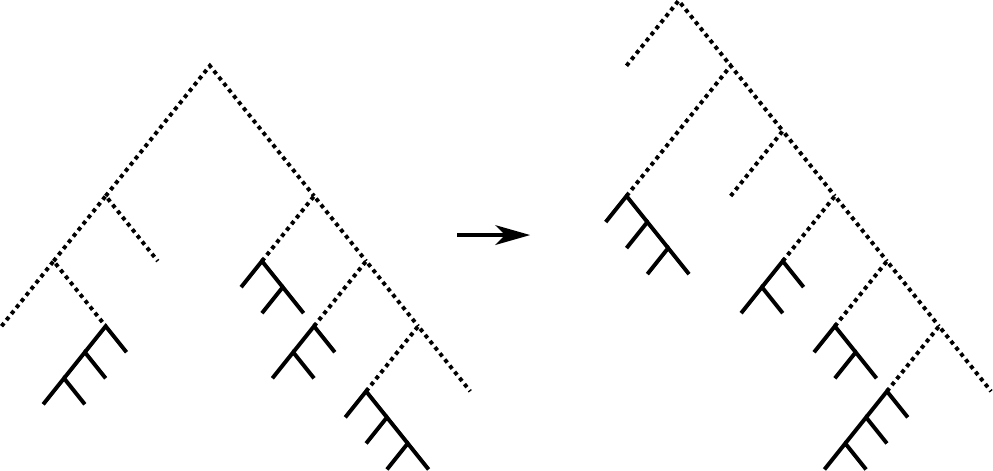}
\end{center}
\caption{A tree pair of $a_0^{-1}a_1^{1}a_2^{-2}a_{-2}^{2}t^2$. The dotted tree pair is a non-reduced tree pair of $x_0^2$. }
\label{ZwrZ_2}
\end{figure}
\end{example}
\begin{example}[the $3$-colorable subgroup $\cF$ and its generalization $\cF_p$] \label{example_pcol}
In \cite{jones2018nogo}, Jones introduced the subgroup $\cF$ called $3$-colorable subgroup. 
The authors \cite[Theorem 1.2]{kodama20223} showed that every nontrivial element of $\cF$ gives a $3$-colorable knot or link. 
Furthermore, for any odd integer $p \geq 3$, the authors \cite{kodama2023p} introduced the subgroup $\cF_p$, which is a generalization of $\cF$ and satisfies $\cF_3=\cF$. 
Similar to $\cF$, every nontrivial element of $\cF_p$ gives a $p$-colorable knot or link \cite[Theorem 3.8]{kodama2023p}. 
Therefore Alexander's theorem does not hold for $\cF$ and $\cF_p$. 
Note that for any Mersenne number $p \geq3$, it is not known (even if $p=3$) whether $\cL|_{\cF_p}\colon\cF_p \to \{ \text{all $p$-colorable knots and links}\}$ is surjective or not. 
See \cite[Question 4.7]{kodama2023p} for details. 
\end{example}
\begin{question}
Let $G$ be one of the subgroups of $F$ described in Examples \ref{example_commutator}, \ref{example_FwrZ}, \ref{example_oriented}, \ref{ex_x0}, and  \ref{example_ZwrZ}. 
Let $\phi\colon G \to F$ be the embedding given by the inclusion. 
Then we already know the image $\cL(\phi(G))$. 
Let $\hat{\phi}\colon G \to F$ be another embedding. 
As mentioned in Remark \ref{1/2_case}, it generally depends on the embedding $\hat{\phi}$ whether Alexander's theorem holds or not. 
What can we say about the image $\cL(\hat{\phi}(G))$?
\end{question}
\begin{question}
Let $G$ be a subgroup of $F$. 
What is the necessary and sufficient condition for $G$ to satisfy Alexander's theorem?
\end{question}

\section*{Acknowledgements}
We would like to thank Professor Tomohiro Fukaya who is the first author's supervisor for his helpful comments. 
We also wish to thank Professor Takuya Sakasai who is the second author's supervisor for his helpful comments.

\bibliographystyle{plain}
\bibliography{bib1} 
\bigskip
\address{
DEPARTMENT OF MATHEMATICAL SCIENCES,
TOKYO METROPOLITAN UNIVERSITY,
MINAMI-OSAWA HACHIOJI, TOKYO, 192-0397, JAPAN
}

\textit{E-mail address}: \href{mailto:kodama-yuya@ed.tmu.ac.jp}{\texttt{kodama-yuya@ed.tmu.ac.jp}}

\address{GRADUATE SCHOOL OF MATHEMATICAL SCIENCES, THE
UNIVERSITY OF TOKYO, 3-8-1 KOMABA, MEGURO-KU, TOKYO, 153-8914,
JAPAN}

\textit{E-mail address}: \href{mailto:takano@ms.u-tokyo.ac.jp}{\texttt{takano@ms.u-tokyo.ac.jp}}
\end{document}